\documentclass{article}
\usepackage[T1]{fontenc}
\usepackage{graphicx}
\usepackage[utf8]{inputenc}
\usepackage{booktabs}
\usepackage[ams,todos]{CMImacros}
\usepackage{amsmath}
\usepackage{amsfonts}
\usepackage{amssymb}
\usepackage{array}
\usepackage{booktabs}
\usepackage{amsthm}
\usepackage{enumerate}
\usepackage{cleveref}
\usepackage{orcidlink}
\usepackage{cleveref}
\usepackage{hyperref}
\hypersetup{
    breaklinks=true,
    colorlinks=true,     % <-- use colored text instead of boxes
    linkcolor=blue,      % <-- color for internal links (sections, TOC)
    citecolor=blue,      % <-- color for citations
    urlcolor=blue
}

\usepackage{titlesec}

\titleformat{\subsubsection}
  {\normalfont\scshape}  % small caps, no bold/italic/underline
  {}                     % no label/numbering
  {1em}                  % horizontal space before heading text
  {}                     % no prefix code
  [\vspace{0.5em}]       % space after heading to break the line

\titlespacing*{\subsubsection}
  {0pt}{1.5em}{0.5em}    % {left}{before}{after} spacing

\usepackage[format=plain,justification=justified]{caption}
\usepackage[format=plain,justification=justified]{caption}

\usepackage[backend=bibtex,maxbibnames=99,firstinits=true,style=numeric,url=false,sorting=none]{biblatex}

\DeclareFieldFormat[article]{title}{#1} % print title literally
\DeclareFieldFormat{title}{#1}
\DeclareFieldFormat{journaltitle}{#1}
\DeclareFieldFormat{booktitle}{#1}
\DeclareFieldFormat{pages}{#1}
\DeclareFieldFormat[inproceedings]{title}{#1}

\DeclareNameAlias{author}{given-family}
\renewcommand*{\labelnamepunct}{\addcolon\space}

\renewbibmacro{in:}{}

\AtEveryBibitem{%
  \clearfield{month}%
  \clearfield{isbn}%
  \clearfield{issn}%
  \clearfield{address}%
  \clearfield{note}%
  \clearfield{language}%
  \ifentrytype{article}{}{%
    \clearfield{publisher}%
  }%
}

\renewbibmacro*{title}{%
  \ifentrytype{article}{%
    \printfield{title}%
  }{%
    \printfield[title]{title}%
  }%
}

\renewbibmacro*{title}{%
  \ifentrytype{thesis}{%
    \printfield{title}%
  }{%
    \printfield[title]{title}%
  }%
}

\DeclareFieldFormat{doi}{%
  \adddot\space
  \url{https://doi.org/#1}%
}

\DeclareBibliographyDriver{article}{%
  \usebibmacro{bibindex}%
  \usebibmacro{begentry}%
  \printnames{author}%
  \setunit{\labelnamepunct}\space
  \usebibmacro{title}%
  \adddot\space
  \printfield{journaltitle}%
  \setunit{\space}%
  \mkbibbold{\printfield{volume}}%
  \setunit{\addcomma\space}%
  \printfield{pages}%
  \setunit{\space}%
  \printtext[parens]{\printfield{year}}%
  \newunit\newblock
  \iffieldundef{doi}{}{%
    \printfield{doi}% <--- MODIFIED
  }%
  \usebibmacro{finentry}%
}

\DeclareBibliographyDriver{book}{%
  \usebibmacro{bibindex}%
  \usebibmacro{begentry}%
  \printnames{author}%
  \setunit{\labelnamepunct}\space
  \printfield{title}%
  \adddot\space
  \printlist{publisher}%
  \setunit{\addcomma\space}%
  \printlist{location}%
  \setunit{\space}%
  \printtext[parens]{\printfield{year}}%
  \usebibmacro{finentry}%
}

\DeclareBibliographyDriver{incollection}{%
  \usebibmacro{bibindex}%
  \usebibmacro{begentry}%
  \printnames{author}%
  \setunit{\labelnamepunct}\space
  \printfield{title}%
  \adddot\space
  \printtext{In: }%
  \printnames{editor}%
  \setunit{\addcomma\space}%
  \printtext{(eds.) }%
  \printfield{booktitle}%
  \setunit{\addcomma\space}%
  \printfield{pages}%
  \adddot\space
  \printlist{publisher}%
  \setunit{\addcomma\space}%
  \printlist{location}%
  \setunit{\space}%
  \printtext[parens]{\printfield{year}}%
  \usebibmacro{finentry}%
}

\DeclareBibliographyDriver{misc}{%
  \usebibmacro{bibindex}%
  \usebibmacro{begentry}%
  \printnames{author}%
  \setunit{\labelnamepunct}\space
  \printfield{title}%
  \adddot\space
  \iffieldundef{journal}{%
    \iffieldundef{eprint}{}{%
      \printtext{arXiv:\space\printfield{eprint}}%
    }%
  }{%
    \printfield{journal}%
  }%
  \setunit{\space}%
  \printtext[parens]{\printfield{year}}%
  \usebibmacro{finentry}%
}

\DeclareBibliographyDriver{inproceedings}{%
  \usebibmacro{bibindex}%
  \usebibmacro{begentry}%
  \printnames{author}%
  \setunit{\labelnamepunct}\space% <--- Uses the new definition (now just a space)
  \printfield{title}%
  \adddot
  \setunit{\addspace}\space
  \printtext{In: }%
  \printfield{booktitle}%
  \setunit{\addcomma\space}%
  \printfield{pages}%
  \setunit{\adddot\space} % Use a period before the publisher/year block
  \printlist{publisher}%
  \setunit{\space}%
  \printtext[parens]{\printfield{year}}%
  \usebibmacro{finentry}%
}

\DeclareFieldFormat[phdthesis]{title}{#1}
\DeclareFieldFormat[thesis]{title}{#1}
\DeclareBibliographyAlias{thesis}{phdthesis}

\DeclareBibliographyDriver{phdthesis}{%
  \usebibmacro{bibindex}%
  \usebibmacro{begentry}%
  \printnames{author}%
  \setunit{\labelnamepunct}\space
  \printfield[titlecase]{title}%
  \adddot\space
  \printfield{type}%
  \space
  \printfield{school}%
  \printtext[parens]{\printfield{year}}%
  \usebibmacro{finentry}%
}

\usepackage{authblk}
\usepackage[margin=1.2in]{geometry} 

\newcommand{\cfeldesy}{Center for Free-Electron Laser Science CFEL, Deutsches
      Elektronen-Synchrotron DESY, Notkestr. 85, 22607 Hamburg, Germany}%

\newcommand{\uhhmaths}{Department of Mathematics, Universität 
Hamburg, Bundesstr. 55,
      20146, Hamburg, Germany}%

      \newcommand{\uhhphys}{Department of Physics, Universität 
Hamburg, Luruper Chaussee 149, 22761 Hamburg, Germany}%

\newtheorem{theorem}{Theorem}[section]
\newtheorem{lemma}{Lemma}[section]

\theoremstyle{definition}

\theoremstyle{example}
\newtheorem{example}{Example}[section]

\theoremstyle{corollary}

\theoremstyle{theorem}
\newtheorem{remark}{Remark}[section]

\numberwithin{equation}{section}

\bibliography{extracted}%

\title{Enhancing polynomial approximation of continuous functions by composition with homeomorphisms}%
% ...
\author[1,2]{Álvaro Fernández Corral\thanks{Corresponding Author: alvaro.fernandez@robochimps.com}\orcidlink{0009-0009-5727-5578}}

\author[3]{Yahya Saleh\orcidlink{0000-0002-3235-217X}}
\affil[1]{\small \cfeldesy}
\affil[2]{\small \uhhphys}
\affil[3]{\small \uhhmaths}
\date{}
\begin{document}
\maketitle

\begin{abstract}\noindent%
We enhance the approximation capabilities of algebraic polynomials by composing
them with homeomorphisms. This composition yields families of functions that
remain dense in the space of continuous functions, while enabling more accurate approximations. For univariate continuous
functions exhibiting a finite number of local extrema, we prove that there exist a
polynomial of finite degree and a homeomorphism whose composition approximates
the target function to arbitrary accuracy. The construction is especially
relevant for multivariate approximation problems, where standard numerical
methods often suffer from the curse of dimensionality. To support our
theoretical results, we investigate both regression tasks and the construction of
molecular potential–energy surfaces, parametrizing the underlying homeomorphism
using invertible neural networks. The numerical 
experiments show strong agreement with our theoretical analysis.

\end{abstract}



\section{Introduction}
Let $\Omega \subset \mathbb{R}$ be a compact and connected set and \( h: \Omega
\to \mathbb{R} \) be a homeomorphism onto its image $\Omega_h:=h(\Omega)$. We denote by
$C(\Omega_h)$ the space of real-valued continuous functions on $\Omega_h$ and
endow it with the standard supremum norm. The function $h$ induces
a bounded composition operator 
$$
K_h: C(\Omega_h) \to C(\Omega), \quad K_h: f \mapsto f \circ h,
$$

see, \eg, ~\textcite{Singh:CompOp1993}. Let \( \Phi :=
\text{span}\left(\{\phi_i\}_{i=0}^\infty\right) \) be a dense set in \( C(\Omega_h)
\) with respect to the supremum norm.
In this work, we study
the density of the set 
	\begin{equation}
		\label{eq:induced_set}
		\Phi_h := \text{span}\left(\{ \phi_i \circ h \}_{i=0}^\infty \right)
	\end{equation}
	in \( C(\Omega) \) with respect to the supremum norm. We demonstrate that
	$h$ being a homeomorphism is sufficient for the density of $\Phi_h$
	in $C(\Omega)$, see \Cref{th:dense_set}.
	We restrict our analysis to the set of
	polynomials $\Pi := \text{span} \left(\{x^i\}_{i=0}^\infty \right)$ and study
	the approximation properties of the induced set
	\begin{equation}
		\label{eq:induced_polynomial_set}
		\Pi_h := \text{span}\left(\{ x^i \circ h \}_{i=0}^\infty \right).
	\end{equation}

In particular, for any function $f \in C(\Omega)$ exhibiting $M\in
\mathbb{N}$ local extrema, we demonstrate the
existence of a homeomorphism $h$ inducing a composition operator $K_h$ and a polynomial of degree $M+1$, denoted by $p$, such that the function
\begin{equation}
	\label{eq:induced_approximant}
	K_h (p) = p \circ h  	
\end{equation}
approximates $f$ arbitrarily well, see \Cref{th:Nextreme}. This polynomial degree is minimal for obtaining arbitrarily accurate approximations to $f$, see~\Cref{Th:MinimalDeg}.

This result can be viewed as a generalization of the classical fact that an exact 
representation of a function by a finite-degree polynomial is only possible if the 
function is itself a polynomial. Here, the expressivity of the set of monomials is 
enhanced by composition with
a function $h$, enabling finite-dimensional 
representations for a much broader class of continuous functions.

To support our theoretical findings, we present numerical experiments in which
the homeomorphism $h$ is modelled by an invertible residual
neural network (iResNet)~\cite{Behrmann:ICML2019:573}. Using univariate
fitting problems, the experiments directly validate our theoretical results
and illustrate the role of the inducing function~$h$. The results demonstrate
orders-of-magnitude improvements in accuracy when using approximants of the
form~\eqref{eq:induced_approximant}.

Although our theory is developed for the
univariate case, we also discuss potential extensions to multidimensional
approximation problems and provide numerical examples. Concretely, a
multivariate dense set induced by composing a polynomial with a homeomorphism is used for fitting potential
energy surfaces, which are of importance in molecular physics,  see~\Cref{sec:multiD}. The
multidimensional results likewise show accuracy improvements by several orders
of magnitude.

Our study is motivated by recent computational developments in
nuclear-motion theory and condensed matter physics. In
\textcite{Saleh:JCTC21:5221}, we proposed to compose basis sets of the $L^2$
space with diffeomorphisms, modelled by iResNets, and used the resulting induced basis sets to compute
vibrational spectra of molecules. Results showed orders-of-magnitudes
improvement over the use of standard basis sets. Further studies extended our
understanding of such induced basis sets~\cite{Vogt:JCP163:154106} and their applicability to higher
dimensional molecular systems~\cite{Zhang:JCP161:024103} and to condensed matter 
physics~\cite{Zhang:PR134:246101,Xie:JML1:3859}. 

Despite the practical success of basis sets induced by composition, little
attention was given to their analysis. Some basic functional and measure-theoretic aspects
of these basis sets were
reported~\cite{Saleh:CAOT20:21,Saleh:PAMM23:e202200239} and some results
on the gained approximation power were derived~\cite{Saleh:thesis:2023}.
However, all these studies considered the $L^2$ space. To the best of our
knowledge, this work is the first to consider the density and approximation
properties of sets of the forms \eqref{eq:induced_set} and
\eqref{eq:induced_polynomial_set} in spaces of continuous functions.

More generally, our results fall in a broader class of literature that investigates
generalized orthogonal polynomials. A prominent example of such studies is
the Müntz-Szász theorem~\cite{Muntz:1914}, 
which characterizes when a set of monomials
with fractional exponents can approximate any continuous function on a compact
interval~\cite{Almira:SAT3:152194}. The Müntz-Szász theorem has seen
various applications, such as, \eg, in spectral methods for solving partial
differential equations~\cite{Shen:SJSC38:A2357,Zeng:arxiv2025.15538}. In
contrast to the Müntz--Szász theorem and its
variations, the induced sets we consider, \ie, \eqref{eq:induced_set} and
~\eqref{eq:induced_polynomial_set}, are constructed by
applying the same transformation to all functions in the dense set. 

\section*{Notation and terminology}

Throughout this paper we denote by $\Omega$ a compact domain of the real line.
We denote by $\Omega^\circ$ the interior of $\Omega$. A function \( f: \Omega \to \mathbb{R} \) is said to
have a local minimum at \(
x^* \in \Omega^\circ \) if there exists \( \varepsilon > 0 \) such that \( f(x) \geq
f(x^*) \) for all \( x \) with \( |x - x^*| < \varepsilon \). If the inequality is
strict, \ie,  \( f(x) > f(x^*) \)  for all \( x \) with \( 0<|x - x^*| < \varepsilon \), 
the function is said to have a strict local minimum
at \( x^* \). The point \( x^* \) is called a local minimizer of $f$. Similarly, \( f
\) is said to have a local maximum at \( x^* \in \Omega^\circ \) if there exists \(
\varepsilon > 0 \) such that \( f(x) \leq f(x^*) \) for all \( x \) with \( |x -
x^*| < \varepsilon \). If the inequality is strict, \ie, \( f(x) < f(x^*) \) for all \( x 
\) with \( 0<|x - x^*| < \varepsilon \), the
function is said to have a strict local maximum at \( x^* \). The point \( x^*
\) is called a local maximizer. A local minimum or maximum of $f$ is called a
local extremum. 

Functions can have a continuum of local extrema when they have an interval of 
constancy, \ie, when they are constant over some interval.  To capture these 
functions, we reformulate the definition of local extrema using
sets. 
A non-empty closed set $X^* \subset \Omega$ is said to be a local minimizer set
of $f$ if $f$ is constant on $X^*$, \ie, if $f(x^*) =
c$ for all $x^* \in X^*$, and there exists $\varepsilon>0$ such that $f(x) >
c$ for all
$x \in \Omega \setminus X^*$ with
$\text{dist}(x,X^*)<\varepsilon$. The same
rationale can be followed to define a set of local maximizers.

An example of a function with a local minimizer set is given in equation 
\eqref{eq:f3} and plotted in \Cref{fig:minfty}.

For a function $f \in C(\Omega)$, we define its critical set as the sequence of
sets of local extrema and denote it as $\mathbb{X} := \{ X^*_{i}\}_{i=1}^M$ for some $M\in \mathbb{N}$.
We assume that the elements of the critical set are always sorted such that, if $x^*_i \in X^*_i$ and $x^*_j \in X^*_j$,
then $x^*_i < x^*_j$ for all $i < j$.  By definition, the sequence of local
extrema  $\{f_i = f(x^*_i)\}_{i=1}^M$
satisfies either 
$$
f_1 < f_2 > f_3 < \dots \ \quad \text{or} \quad \
f_1 > f_2 < f_3 > \dots
$$
for any $x_i^* \in X_i^*$.
These two alternating relations can be written compactly as 
$$
\pm (-1)^i (f_{i+1} - f_{i}) > 0, \quad \text{for }i=1, \dots, M-1.
$$

A function $f \in C(\Omega)$ is said to have an interval of constancy $I 
\subseteq \Omega$ if $I$ is a measurable set on which $f$ is constant, \ie, 
$f(x) = c$ for some real number $c$ and all $x \in I$. By definition, sets of non-strict 
local extrema can only occur in intervals of constancy.

A
function $h: \Omega \to \mathbb{R}$ is said to be a homeomorphism onto its image $\Omega_h = h(\Omega)$
if $h$ is continuous,
injective, and its inverse $h^{-1}:\Omega_h \to \Omega$ is continuous. This
is equivalent to $h$ being continuous and strictly monotonic. The critical set of 
a
homeomorphism is empty. 
\section{Preliminaries}
We recall several classical results from approximation theory that help place 
our work in context. 

The celebrated Weierstrass approximation theorem states that the set of polynomials $
\Pi$
is dense in \( C(\Omega) \) with respect to the supremum norm. In other words,
for any \( f \in C(\Omega) \) and \( \varepsilon > 0 \), there exist an integer
\( d \in \mathbb{N} \) and a polynomial  
\[
p_d(x) = \sum_{i=0}^{d} a_i x^i
\]
such that
\[
\sup_{x \in \Omega} |f(x) - p_d(x)| < \varepsilon.
\]
The integer \( d \) is referred to as the degree of the approximating polynomial \( p_d \).

An exact representation of \( f \in C(\Omega) \) by a finite-degree polynomial is possible only if \( f \) itself is a polynomial. In general, polynomial approximations converge to \( f \) asymptotically as \( d \to \infty \), with the convergence rate determined by the smoothness of \( f \). Classical results such as Jackson’s inequality provide quantitative bounds on this approximation error in terms of the function’s regularity~\cite{Lloyd:AproxTheory:2019}. 

The study of other dense sets in \( C(\Omega) \) has a long history in
approximation theory. The most notable example is provided by the Müntz-Szász
theorem~\cite{Almira:SAT3:152194}.  For \( \Omega = [a,b] \) with \( b > a > 0
\) consider 
\[
\Psi_\Lambda = \mathrm{span}\!\left(\{x^{\lambda_i}\}_{i=0}^\infty\right),
\]
where \( \Lambda = \{\lambda_i\}_{i=0}^\infty \subseteq \mathbb{R}_+ \). Then \( \Psi_\Lambda \) is dense in \( C(\Omega) \) if and only if \( \lambda_0 = 0 \) and 
\[
\sum_{i=1}^\infty \frac{1}{\lambda_i} = \infty.
\]
By setting \( \Lambda = \mathbb{N}
\) one recovers the set of polynomials $\Pi$.

\section{Approximation \emph{via} polynomials composed with homeomorphisms}\label{sec:1Dtheory}

In this section, we investigate the improved approximation capabilities achieved
by composing dense sets with homeomorphisms.
We begin by establishing the theoretical guarantees for the density of such
composed sets in the space of continuous functions.

\begin{theorem}\label{th:dense_set}
	Let $\Omega \subset \mathbb{R}$ be a connected and compact set and let \( h:
	\Omega \to \mathbb{R} \) be a homeomorphism onto its image $\Omega_h$.
	Let \( \Phi := \text{\emph{span}}\left(\{\phi_i\}_{i=0}^\infty\right) \) be dense in \( C(\Omega_h) \) with respect to the supremum norm. Then the set
	\[
		\Phi_h := \text{\emph{span}}\left(\{ \phi_i \circ h \}_{i=0}^\infty \right)
	\]
	is dense in \( C(\Omega) \) with respect to the supremum norm.
\end{theorem}

\begin{proof}
Let \( f \in C(\Omega) \) be arbitrary. Since $h$ is a homeomorphism,  it admits a continuous inverse \( h^{-1}: \Omega_h \to \Omega \)~\cite{Apostol:MathAnalyis:1974}.

The composition of continuous functions is continuous~\cite{Rudin:PrinciplesMathAnalysis:1976, Apostol:MathAnalyis:1974}, so \( f \circ h^{-1} \in C(\Omega_h) \). By the density of \( \Phi \) in  \( C(\Omega_h) \), for any \( \varepsilon > 0 \), there exists a function \( g \in \Phi \) such that
\[
\sup_{y \in \Omega_h} \left| f \circ h^{-1}(y) - g(y) \right| < \varepsilon.
\]
Define the function \( \tilde{g} := g \circ h \). Since \( g \in \Phi = \text{span}\left(\{\phi_i\}_{i=0}^\infty \right)\), it follows that \( \tilde{g} \in \Phi_h = \text{span}\left(\{\phi_i \circ h\}_{i=0}^\infty \right) \). For all \( x \in \Omega \) and $y = h(x) \in \Omega_h$, we then have
\[
\abs{f(x) - \tilde{g}(x)} = \abs{f(x) - g(h(x))} = \abs{f \circ h^{-1}(y) - g(y)} < \varepsilon,
\]
since \( \Omega_h = h(\Omega) \). Therefore,
\[
\sup_{x \in \Omega} |f(x) - \tilde{g}(x)| < \varepsilon.
\]
This shows that any \( f \in C(\Omega) \) can be approximated arbitrarily well with respect to the supremum norm by functions in \( \Phi_h \), and thus \( \Phi_h \) is dense in \( C(\Omega) \) with respect to the supremum norm.
\end{proof}
Note that \Cref{th:dense_set} remains valid if the domain $\Omega$ is a
compact and connected subset of $\mathbb{R}^d$.

\begin{remark}
	\Cref{th:dense_set} broadens the use of dense sets in approximation theory in two complementary ways. 
First, it allows a dense family of functions defined on one compact domain to be
transferred to any other compact domain \emph{via} a homeomorphism \(h\); this includes the standard case where \(h\) is linear. 
Second, it highlights a distinct mechanism for improving approximation accuracy.
In standard approximation theory, one seeks an approximation of an unknown
function \(f\) in the span of finitely many functions from a dense set. The
accuracy of the approximation can typically be improved by increasing the number
of functions used in the approximation.
In contrast, \Cref{th:dense_set} suggests that one can also improve
the approximation by optimizing the choice of the dense set itself through a suitable function
\(h\). 
This observation motivates considering the family
\[
	H = \left\{ \Phi_h \;\middle|\; h:\Omega \to \mathbb{R}, \; h \text{ is a homeomorphism onto its image $\Omega_h$} \right\},
\]
and searching for the dense set within this family that is most adequate for a given approximation task.
\end{remark}

\begin{remark}
	\Cref{th:dense_set} requires the image of $h$
	to be the domain of the dense set $\Phi$. For a given dense set $\Phi$, this restriction imposes a constraint on the
	possibilities for the transformation $h$. To overcome this limitation,
	sets that are dense on any compact subset of the real line can be used. In
	particular, the set of polynomials~$\Pi$ satisfies this
	condition. This motivates the preferential use of the induced set 
	$$\Pi_h =  \text{span}\left(\left\{x^i \circ h\right\}_{i=0}^\infty\right) =
	\text{span}\left(\left\{ h^i\right\}_{i=0}^\infty\right),$$
	which is dense for every compact domain and every image of $h$. 
\end{remark}

We note that, in the special case in which $\Phi$ satisfies the Müntz-Szász theorem, \Cref{th:dense_set} resembles
a density result reported by~\textcite{Jaming:BdScMat:102933}. Specifically, the
authors showed that the set $\Phi_h = \mathrm{span}\left(\{h^{\lambda_i}\}_{i=0}^\infty\right)$--- where the set of powers $\Lambda = \{\lambda_i\}_{i=0}^\infty$ satisfies the Müntz-Szász conditions--- is
dense in the space of continuous functions only if $h$ is monotonic and its first and second-order derivative do not vanish simultaneously ~\cite[Proposition
3.1]{Jaming:BdScMat:102933}. In contrast, our result does not require $h$ to be differentiable.

In what follows, we restrict our analysis to the study of the approximation properties of the induced polynomial set $\Pi_h$. 
For any continuous function $f$ with finitely-many sets of local extrema, we
demonstrate the existence of a homeomorphism $h$ and a finite-degree polynomial
$p$, such that $p\circ h$ approximates $f$
arbitrarily well. To prove this, we first introduce two
instrumental lemmas. 

\begin{lemma}\label{th:interval}
	For some bounded intervals $\Omega_f,\Omega_g \subset \mathbb{R}$ let $f : \Omega_f
	\to \mathbb{R}$ be a monotonic, continuous function and let $g : \Omega_g
	\to \mathbb{R}$ be a strictly monotonic, continuous function. Suppose
	that $f(\Omega_f) =
	g(\Omega_g)=:I$. Then, for every $\varepsilon > 0$, there exists a
	homeomorphism $h: \Omega_f \to \Omega_g$ such that
	\[
		\sup_{y \in \Omega_f}\abs{f(y) - g \circ h(y)} < \varepsilon.  
	\]
	
If in fact $f$ is strictly monotonic, then
\[
f(y) = g \circ h(y)
\]
for all $y \in \Omega_f$.
\end{lemma}

\begin{proof}
	Using the density of strictly monotonic functions in the space of monotonic functions, there exists a continuous strictly monotonic function $\hat{f}$ such that, for every $\varepsilon > 0$,
	\[
		\sup_{y \in \Omega_f}\abs{f(y) - \hat{f}(y)} < \varepsilon.  
	\]
	
	Since $g$ is a strictly monotonic and continuous function, its inverse
	function $g^{-1}$ exists and is unique, continuous and strictly monotonic~\cite{Apostol:MathAnalyis:1974}. The function $h = g^{-1} \circ
	\hat{f}$ satisfies the conditions of the lemma.	
	
	In particular, if $f$ is strictly monotonic, then $f = \hat{f}$ and the proposed homeomorphism $h$ satisfies $f(y) = g\circ h(y)$ for all $y \in \Omega_f$.
\end{proof}

\begin{lemma}
	\label{Thm:Chandler}
	For some $M\in \mathbb{N}$ let $\{f_i\}_{i=0}^{M+1}$ be a set of real numbers satisfying
	$$
\pm (-1)^i (f_{i+1} - f_{i}) > 0, \quad \text{for }i=0, \dots, M,
$$
 for either positive or negative sign.
	Then, there exists a polynomial $p$ of degree $M+1$, and a set of distinct
	real numbers $y_0 <y_1< \dots <
	y_{M+1}$ such that
	\begin{align*}
		\begin{cases}
			p \ (y_i) = f_i \quad &i = 0, \dots M+1, \\
			p'(y_i)=0 \quad &i = 1, \dots, M.
		\end{cases}
	\end{align*}
In particular, $\{y_i\}_{i=1}^{M}$ is the critical set of $p$.

\end{lemma}

\begin{proof}
	\textcite{Chandler:TAMM64:679680} proved that a polynomial $p$ and interior
	points $\{y_i\}_{i=1}^M$ that satisfy the conditions of the theorem exist.
	The proof involves constructing a function $\phi$ that maps the interior
	points $\{y_i\}_{i=1}^M$ to the set 
	$$\{(-1)^{i-1} \left(\hat p(y_{i+1}) - \hat p(y_i)\right)\}_{i=1}^M,$$
	 where $\hat p$ is the polynomial that has $\{y_i\}_{i=1}^M$ as its critical point. The determinant of $\phi$ is non-zero if the recurrent relation of the values $f_i$ is satisfied, thereby proving the existence of $p$. 
	 
	The existence of exterior points $y_0$ and $y_{M+1}$ can be proven following a logic that depends on the alternating relation of $f_i$. If $f_0 < f_1$, $p$
	is increasing on $(-\infty, y_1)$, so by continuity there exists a $y_0 < y_1$ such that 
	$p(y_0) = f_0$. Alternatively, if $f_0 > f_1$, then $p$
	is decreasing on $(-\infty, y_1)$, and there exists a $y_0 < y_1$ such that 
	$p(y_0) = f_0$. Similarly, we can deduce the existence of a $y_{M+1} > y_M$ such that 
	$p(y_{M+1}) = f_{M+1}$.
\end{proof}
With these technical results in hand, we can now proceed to state and demonstrate
the main theorem of our work. We omit the trivial case in which $f$ is constant,
since any constant function can already be represented exactly by a degree-zero
polynomial. 
\begin{theorem}\label{th:Nextreme}For any compact $\Omega\subset 
\mathbb{R}$ let \( f : \Omega \to \mathbb{R} \) be a continuous 
non-constant function. Suppose that \( f \) has 
	exactly \( M \in \mathbb{N} \) sets of local extrema $\{X_i\}_{i=1}^M$ on 
	\( \Omega \). Then, for any \( \varepsilon > 0 \), there exists
	\begin{itemize}
		\item a compact interval \( \Omega_h \subset \mathbb{R} \),
		\item a homeomorphism \( h : \Omega \to \Omega_h \),
		\item and a polynomial $p$ of degree $M+1$ with $M$ distinct strict local extrema,
	\end{itemize}
	such that 
	\begin{align}\label{eq:theorem_dense}
	\sup_{x \in \Omega} \ \abs{f(x) - p \circ h(x)} < \varepsilon
	\end{align}
	
	holds. 
	
	If, in addition, $f$ has no interval of constancy, 
	then there exist $\Omega_h$, $h$, and $p$ as above such that
	\begin{align}\label{eq:theorem_equals}
		f(x) = p \circ h(x)
	\end{align}
	for all $x \in \Omega$.
\end{theorem}

\begin{proof}

Let $\Omega := [x_0, x_{M+1}]$ with $x_0 < x_{M+1}$.	Choose one 
representative
$x_i$	of each set of local extremum  $X_i$ arbitrarily. Let the sets in
	$\mathbb{X}$ be sorted such that \( x_0 < x_1 < x_2 < \dots < x_{M+1} \).
	Set $I_0 := [x_0, x_1]$ and \( I_i := (x_i, x_{i+1}] \) for all $1 \leq i \leq M$. Note that the restrictions \( f|_{I_i} \) are either
	monotonic or strictly monotonic. 
	
	Denote the functional values at the local extrema and endpoints  by $f_i = f(x_i)$ for $i=0, \dots,
	M+1$. Since these values contain the evaluations at extreme points and endpoints, they follow
	the recurrent relation $ \pm (-1)^i (f_{i+1} - f_i) > 0$ for all $i = 0,
	\dots, M$, either for the positive or the negative sign.
	
	By \Cref{Thm:Chandler}, there exist real numbers \( y_0 < y_1 < \dots <
	y_{M+1} \) and a polynomial
	$p$ of degree $M+1$ such that  
	\begin{align}\label{eq:th_interp}
		\begin{cases}
		p \ (y_i) = f_i \quad &i = 0, \dots M+1, \\
		p'(y_i)=0 \quad &i = 1, \dots, M.
		\end{cases}
	\end{align}

	Set $\Omega_h := [y_0,
	y_{M+1}]$, $J_0 := [y_0, y_1]$ and \( J_i := (y_i, y_{i+1}] \) for all $1 \leq i \leq M$.
	As $p$ is of degree $M+1$ and satisfies~\eqref{eq:th_interp}, the restrictions \( p|_{J_i} \) are strictly
	monotonic.		
	In addition, this construction imposes that the images of the restrictions
	are shared $f(I_i) = p(J_i)$ for all $i =0, \dots, M$. By
	applying~\Cref{th:interval} to each restriction, for every positive constant
	$\varepsilon>0$, there exist homeomorphisms $h_i : I_i \to J_i$ such that
	\[
	\sup_{x \in I_i} \ \abs{f(x) - p \circ h_i(x)} < \varepsilon, \quad \text{for all } i = 0, \dots, M.
	\]
	
	Define the global function \( h: \Omega \to \Omega_h \) piecewise by
	\[
	h(x) := h_i(x), \quad \text{for } x \in I_i, \quad i = 0, \dots, M.
	\]
	Since each \( h_i \) is continuous and invertible, and the intervals \( 
	I_i \) only meet at endpoints, the concatenation defines a 
	homeomorphism \( h : \Omega \to \Omega_h\).
	
	It follows that, for every positive constant $\varepsilon>0$, there exists a
	continuous invertible function $h : \Omega \to \Omega_h$ that satisfies
	\[
	\sup_{x \in \Omega} \ \abs{f(x) - p \circ h(x)} = \max_{0\leq i \leq M}  \left( \sup_{x \in I_i} \abs{f(x) - p \circ h_i(x)}\right) < \varepsilon,
	\]
	completing the proof of \eqref{eq:theorem_dense}. 
	
	Last, if $f$ has no interval of constancy, its critical set $\mathbb{X}$ is constituted by strict local extrema and is a sequence of points. Additionally, every restriction $f|_{I_i}$ is strictly monotonic. It follows
	from the proof of~\Cref{th:interval} that $f(x) = p \circ h_i(x)$ for all
	$x\in I_i$  and every $i = 0, \dots, M$. Therefore, the approximation can be made exact,
	proving \eqref{eq:theorem_equals}.
\end{proof}

\Cref{th:Nextreme} establishes the existence of a homeomorphism $h$ and a
corresponding induced set that allows a finite dimensional arbitrary-precision
approximation of a broad class of functions. The theorem demonstrates that one
can increase the accuracy of an approximation by finding a suitable induced set $\Pi_h$,
rather than by increasing the degree of the polynomial within a fixed induced set.

To gain intuition on the approximation within the induced set $\Pi_h$, we consider the case of functions with a single local
extremum. In such case, the function $h$ can be constructed analytically. 
\begin{example}
	\label{example:closed-form}
	Let $f \in C(\Omega)$ be a continuous function with a single set of local extremum $X_0$ and let $x_0 \in X_0$.  
	By~\Cref{th:Nextreme}, there exists a strictly monotonic function $h$ and polynomial $p$ of degree $2$ such that  
	\[
		\sup_{x \in \Omega} | f(x) - p \circ h(x) | < \varepsilon \qquad \text{ for all }  \varepsilon>0.
	\]
	
	We now derive the real coefficients $\{a_i\}_{i=0}^2$ and transformation $h$ such that 
	\[
	p \circ h = \sum_{i=0}^2 a_i h^i
	\]
	approximates  $f$ arbitrarily well.

	The zero-order term is unaffected by the composition. Set $a_0 = f(x_0)$ and define the remainder of the approximation $\hat{f} = f - a_0$.  
	The image of $\hat{f}$ is either nonnegative or nonpositive on $\Omega$.  As
	it will be shown, the first-order term is unnecessary to capture the shape
	of a function with a single extremum, and therefore we set $a_1 = 0$.  
	The approximation problem reduces to approximating $\hat{f}$ by $a_2 h^2$.
	The second-order coefficient is chosen as $a_2 \in \{+1,-1\}$ so that $\hat{f}/a_2 \geq 0$.  
	
	This leads to the representation  
	\begin{align}\label{eq:T1_func}
		\hat{h}(x) = \operatorname{sign}(x-x_0)\,\sqrt{|\hat{f}(x)|},
	\end{align}  
	which is continuous since the sign discontinuity occurs exactly at $x_0$, where $\hat{f}(x_0)=0$.  
	
	In general, $\hat{h}$ needs not be strictly monotonic (as no restriction was imposed to $\hat{f}$ for this to be true).  
	Since strictly monotonic functions are dense in the space of monotonic 
	functions, for any $\varepsilon>0$ there exists a strictly monotonic function 
	$h$ satisfying  
	\[
	\sup_{x \in \Omega} |h(x) - \hat{h}(x)| < \varepsilon.
	\]  
	
	Therefore, the composition of $h$ with the second-degree polynomial with coefficients $(a_0,0,a_2)$, provides an approximation of $f$ to arbitrary precision.  
\end{example}

Next, we demonstrate that the degree of the polynomial $p$ in~\Cref{th:Nextreme} is the minimal degree for obtaining an arbitrarily accurate approximation.
\begin{theorem}\label{Th:MinimalDeg}
	For any compact $\Omega \subset \mathbb{R}$ let $f:\Omega \to 
	\mathbb{R}$ be a continuous non-constant function with $M\in \mathbb{N}$ sets of local extrema. If for all 
	$\varepsilon>0$, there exist 
	some homeomorphism $h: \Omega \to \Omega_h$ and some polynomial $p: \Omega_h \to \mathbb{R}$  of 
	degree $d$ satisfying 
	$$
	\sup_{x \in \Omega} \ |f(x) - p \circ h (x)|<\varepsilon,
	$$ 
	then $d \geq M+1$. In other words, the degree of the polynomial $p$ in 
	\Cref{th:Nextreme} is minimal.
\end{theorem}

\begin{proof}
By \Cref{th:Nextreme}, for every $\varepsilon > 0$ there exist a homeomorphism
$h : \Omega \to \Omega_h$ and a polynomial $p$ of degree $M+1$ such that
\[
\sup_{x \in \Omega} |f(x) - p \circ h(x)| < \varepsilon .
\]
Consequently, the same approximation property holds for any degree $d \ge M+1$,
since polynomials of degree $M+1$ are contained in the class of polynomials of
degree at most $d$.

It therefore remains to show that no such approximation is possible when
$d < M+1$.

	Set $\Omega= [x_0, 
	x_{M+1}]$ with $x_0 < x_{M+1}$. Let $f$ be continuous with $M$ sets of
	local extrema. Let $\mathbb{X}=\{X_i\}_{i=1}^M$ denote the ordered
	collection of these sets. Choose an arbitrary
	point $x_i 
	\in X_i$ for all $i=1, \dots, M$.
	Set 
	$$
	\Delta_i = |f(x_{i+1} )
	- f(x_i)|	 \quad \text{for all } i=0,\dots, M.
	$$
	
	Since $f$ is non-constant between consecutive extrema and endpoints, we have
$\Delta_i > 0$ for all $i$. Set 
$$
\varepsilon = 
	\frac{1}{2} \min_{0\leq i \leq M} \Delta_i
	$$
	 and let $p$ and $h$ be such that
	$$
	\sup_{x \in \Omega} \ |f(x) - p \circ h (x)|<\varepsilon.
	$$ 
	
	Equivalently, this can be written as
	\begin{align}\label{eq:ThmStatement}
	\sup_{y \in \Omega_h} \ |f \circ h^{-1}(y) - p(y)|<\varepsilon,
	\end{align}
	where $y = h(x)$. 

	Since $h^{-1}$ is a homeomorphism, $g := f\circ h^{-1}$ has the same number of 
	sets of local extrema as $f$. Denote by $\{y_i\}_{i=1}^{M}$ the
	representatives of the sets of local extrema of $g$, that is, $y_i = h(x_i)$.
	Additionally, set $\Omega_h = [y_0, y_{M+1}]$, where $y_0=h(x_0)$ and
	$y_{M+1} = h(x_{M+1})$. Note that, irrespective of $h$, the function $g$ satisfies the recurrent relation
	
	$$
	|g (y_{i+1}) - g (y_{i})|  = \Delta_i \quad \text{for all 
	} i = 0,\dots, M .
	$$
	
	Using the reverse triangular inequality twice, we have that for all $i=0,\dots, M$,
	\begin{align}\label{eq:differences}
		|p(y_{i+1}) - p(y_i)| \geq \Delta_i - |p(y_{i+1}) - g(y_{i+1})| - 
		|p(y_{i}) - g(y_i)| > 2 
		\varepsilon - \varepsilon - \varepsilon = 0,
	\end{align}
	where the last inequality follows from the definition of $\varepsilon$.
	Thus, the values $p(y_i)$ and $p(y_{i+1})$ are distinct for all $i$.
	
	Given that $g(x_i)$ are evaluations at extrema and endpoints, they 
	must satisfy either recurrent relation $\pm (-1)^i \bigl(g(x_{i+1}) - g(x_i)\bigr) >0$ for all 
	$i=0,\dots, M$.  It follows from \eqref{eq:differences} and the choice of 
	$\varepsilon$ that the evaluations of 
	$p(y_i)$ must satisfy the same recurrent relation. 
	
	Therefore, for each $i=0,\dots,M-1$, either
\[
p(y_i) > p(y_{i+1}) < p(y_{i+2})
\quad \text{or} \quad
p(y_i) < p(y_{i+1}) > p(y_{i+2}).
\]
By the extreme value theorem, this implies that $p$ has at least one local extremum
in each interval $(y_i, y_{i+2})$. Hence, $p$ has at least $M$ distinct local
extrema. Therefore, $p$ is a polynomial of degree $d 
	\geq M+1$.  
	
\end{proof}

	The following is an equivalent statement of~\Cref{Th:MinimalDeg}. 

	\begin{theorem}
	For any compact $\Omega \subset \mathbb{R}$ let $f:\Omega \to 
	\mathbb{R}$ be a continuous non-constant function. If for all $\varepsilon>0$, there exist 
	some homeomorphism $h: \Omega \to \Omega_h$ and a polynomial $p: \Omega_h \to \mathbb{R}$  of 
	degree $d$ with $d-1$ distinct local extrema satisfying
	$$
	\sup_{x \in \Omega} \ |f(x) - p \circ h (x)|<\varepsilon,
	$$ 
	then $f$ has $M$ sets of local extrema, with $M \leq d-1$.
	\end{theorem}
	\begin{proof}
		The proof follows by contraposition of~\Cref{Th:MinimalDeg}. If $f$ has more than $d-1$ sets of local extrema, then for some $\varepsilon>0$, there do not exist a homeomorphism $h$ and a polynomial $p$ of degree $d$ such that 
		$$
		\sup_{x \in \Omega} \ |f(x) - p \circ h (x)|<\varepsilon.
		$$
	\end{proof}
	
\section{Numerical results}\label{sec:NumRes}

In this section, we present numerical experiments that support our theoretical results. 
For a given target function $f$, we construct approximations of the form 
\[
	\label{eq:nonl_ansatz}
	\hat{f}_{\Theta} = \sum_{i=0}^{N} a_i h_{\Theta}^i(x),
\]
where the nonlinear function $h_\Theta$ was modeled \emph{via} an iResNet. By construction iResNets are
smooth bi-Lipschitz transformations and hence satisfy the assumptions on $h$ in
\Cref{th:dense_set} and \Cref{th:Nextreme}. However, these theorems assume the class of all homeomorphisms, which include non-smooth functions. For this reason, some of the expected results, such as the exact representation in  \Cref{th:Nextreme}, might not be computationally reproducible. 
The iResNet was built using $15$ residual blocks, each composed of $2$ hidden layers of $8$
units each and that use LipSwish as activation function.

We compare the optimization results with standard polynomial approximations, 
\[
	\label{eq:l_ansatz}
	\hat{f} = \sum_{i=0}^{N} a_i x^i.
\]
We demonstrate that optimizing the transformation $h$ 
enhances approximation of a broad class of continuous target functions with different number of local extrema, as described in \Cref{th:interval} and 
\Cref{th:Nextreme}.

\subsection{Approximating continuous functions with one strict local extremum}
\label{subsec:num_Mextrema}

In Example~\Cref{example:closed-form}, we established that any continuous function with a single strict local
extremum can be approximated by a polynomial of degree $2$
composed with a homeomorphism $h$.  
We now provide numerical examples illustrating this result.  

Consider the target function
\begin{equation}
\label{eq:f1}	
f(x) = \exp(x) + \exp(-x), \qquad x \in \Omega = [-10,10],
\end{equation}
which has a single extremum at $x_0=0$.  

Since the minimum occurs at $x_0=0$, we set 
$a_0 = f(x_0)=2$. The coefficient of the linear term is taken as $a_1=0$, as it is unnecessary for capturing the single-extremum structure. Because the remainder $\hat{f}=f-a_0$ is non-negative, we set $a_2=1$.  
The corresponding transformation $h$ is given analytically by~\eqref{eq:T1_func}, that is
\[
h(x) = \operatorname{sign}(x)\,\sqrt{\exp(x)+\exp(-x)-2}.
\]

In addition to this closed-form construction, we also computed a numerical approximation of the form
\[
\hat{f}(x) = \sum_{i=0}^{2} a_i \, h_{\Theta^*}(x)^i,
\]
where $h_\Theta$ is parametrized by an iResNet.  
Given a training set of $P$ equidistant points 
$D=\{x_p\}_{p=0}^{P-1}$, the neural network parameters $\Theta^*$ were obtained from the optimization problem
\begin{align}\label{eq:Objective_nonl}
	\Theta^*  = \arg\min_{\Theta} 
	\left( \sqrt{\frac{1}{P} \sum_{p = 0}^{P-1} \, \big| f(x_p) - \sum_{i=0}^2 a_i \, h_\Theta(x_p)^i \big|^2} \right).
\end{align}
For this example we used $P = 301$ points. 
For the optimization problem we used the root mean squared error (RMSE) and
not the supremum norm, since the latter is not differentiable. 

This nonlinear optimization was solved using the \texttt{Optax}~\cite{Optax:DeepMindJax} implementation of Adam~\cite{Kingma:ICLR2015}, a gradient-based stochastic optimization algorithm.

As a benchmark, we also optimized a second-degree polynomial in the original input space by solving
\[
\mathbf{X}\mathbf{a} = \mathbf{f},
\]
with $\mathbf{a} = [a_0, a_1, a_2]$, $\mathbf{f} = [f(x_0), \dots, f(x_{300})]$,
and the matrix $\mathbf{X}$ is given by $\mathbf{X}^{ip} = x_p^i$.  

The results are reported in \Cref{fig:expx}.  
Interestingly, the learned transformation $h_\Theta$ closely resembles the analytically derived function $h$, despite no explicit constraint enforcing this behavior.  
For validation, we computed the RMSE over a grid of $P=5001$ equidistant points
$\{\hat{x}_j\}_{p=0}^{P-1}$, obtaining a value of $11.146$. The corresponding mean relative
error (MRE) over the validation points is
\[
\text{MRE} = \frac{1}{M} \sum_{j = 0}^{M-1} \, \abs{\frac{ f(\hat{x}_p) - \sum_{i=0}^2 a_i \, h_{\Theta^*}(\hat{x}_p)^i }{f(\hat{x}_p)}} =  0.021 . 
\]

In comparison, achieving a similar accuracy with direct polynomial fitting
requires a polynomial of degree at least $10$.  

\begin{figure}[h!]
	\centering
	\includegraphics[width=\linewidth]{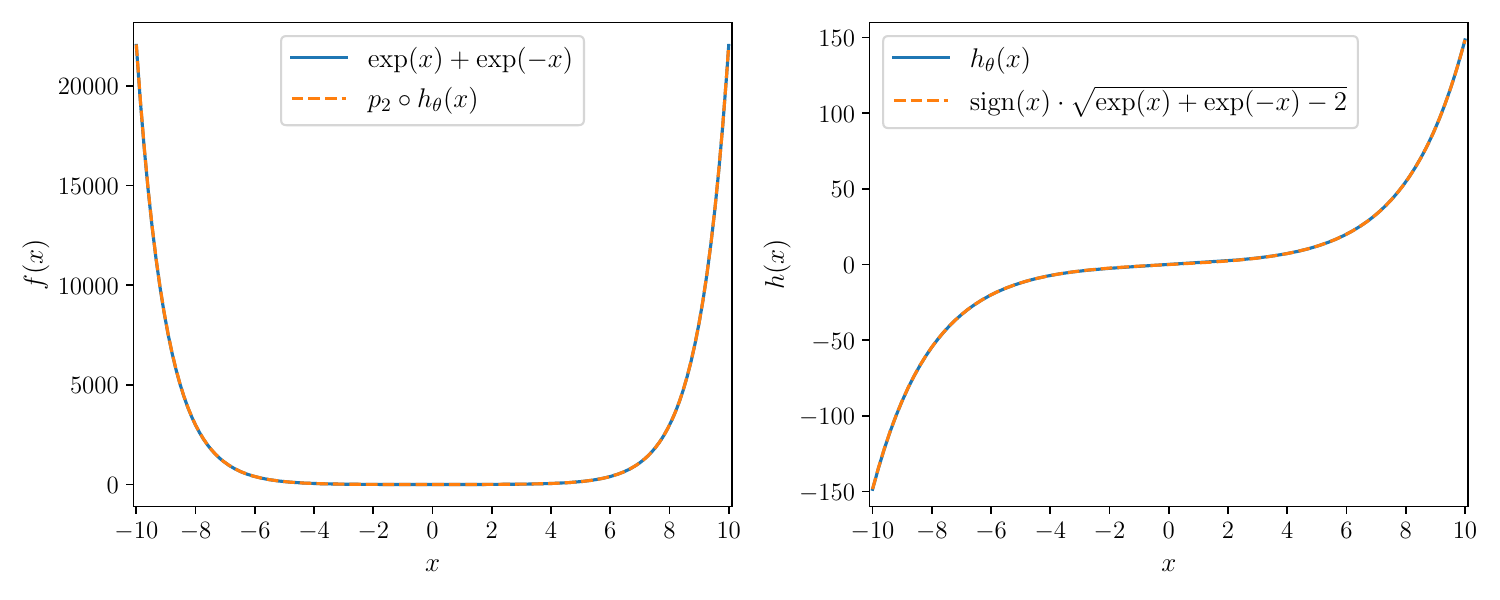}
	\caption{\textbf{Approximation of \eqref{eq:f1} using
	polynomials composed with a homeomorphism.}
		\textit{Left:} Plotted are the target function and the fitted second-order power series expansion $p_2(x) = x^2 +2$ composed with an invertible bi-Lipschitz function $h_\theta$, parametrized by an iResNet with parameters $\theta$.  \textit{Right:} The invertible function that resulted as the solution of the optimization and its comparison with the expected function $h(x) = \text{sign}(x) \cdot \sqrt{f(x) - 2}$.}
	\label{fig:expx}
\end{figure}

\subsection{Approximating non-smooth functions}

Consider the non-differentiable continuous function
\begin{equation}
\label{eq:f2}	
c(x) = 
	\begin{cases}
		1 -(x - 1)^2, \quad &x > 0 \\
		\arctan(-x), \quad &x \leq 0 \\
	\end{cases}
\end{equation}
for $x \in \Omega = [-3,3]$.
This function is continuous over its domain, since both defining
functions are continuous and their values at the discontinuity point match. The
function contains two strict local extrema in its domain, placed at the
discontinuity point $x=0$ and at $x=1$. Therefore, an arbitrarily accurate approximation using
a third-degree polynomial composed with a homeomorphism
$h$ is possible. 

As no analytical choice of $\{a_i\}_{i=0}^3$ is available for this case, the objective of the optimization is to find both the optimal $a_i^*$ and $\Theta^*$. Given a training set of equidistant points 
$D=\{x_p\}_{p=0}^{300}$, the linear coefficients $a_i^*$ and network parameters
$\Theta^*$ were obtained by solving
\begin{align}\label{eq:Objective_nonl2}
	a_i^*, \Theta^* = \arg\min_{a_i, \Theta} 
	\left( \sqrt{\frac{1}{P} \sum_{p = 0}^{P-1} \, \big| f(x_p) - \sum_{i=0}^2 a_i \, h_\Theta(x_p)^i \big|^2} \right).
\end{align}
This nonlinear problem was solved using the following optimization scheme:  
the nonlinear parameters $\Theta$ were updated with Adam. For each choice of $\Theta$, the linear
coefficients $a_i$ were determined by solving the least-squares system
$$
\sum_{i=0}^2 a_i h_\Theta(x)^i \approx f(x), \qquad x \in D,
$$
which amounts to computing the pseudoinverse of the matrix
$$
\mathbf{X}_\Theta^{ip} = h_\Theta(x_p)^i, \qquad i=0,1,2,\;\; p=0,\dots,300.
$$
This pseudoinversion guarantees that the coefficients $a_i$ are optimal for every parameter set $\Theta$.

We illustrate the result of this numerical experiment
in~\Cref{fig:C_fit}. The approximation error is larger in the neighborhood of
$x=0$, since this is the discontinuity point of $f$ and $p_3 \circ h_\theta$ is smooth everywhere by construction ($h_\theta$ is modeled using an iResNet, which is smooth by design).

\begin{figure}[h!]
	\includegraphics[width=\linewidth]{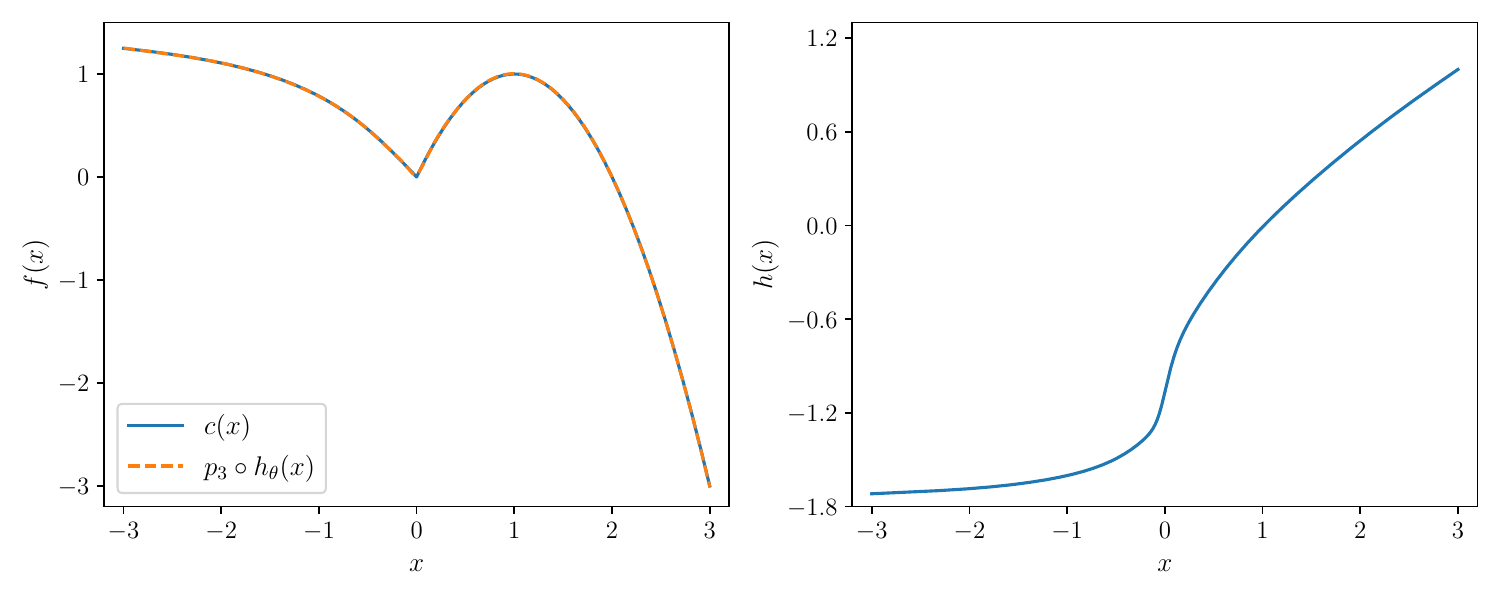}
	\caption{\textbf{Approximation of \eqref{eq:f2} using polynomials composed with a homeomorphism.}\textit{ Left:} The target function and the fitted second-order
	expansion composed with an invertible function. \textit{Right:} The invertible function that resulted as the
	solution of the optimization. }
	\label{fig:C_fit}
\end{figure}

For validation, a polynomial was fitted on the training points and both approaches were evaluated over an equidistant grid of $5001$ points in the domain. In the validation set, the induced set achieved a RMSE of $3.92 \cdot 10^{-3}$ and a maximum absolute error (MAE) of $0.038$. In contrast, the best direct polynomial fit, using a polynomial of degree $80$, yielded 
a RMSE of $6.93 \cdot 10^{-3}$ and a MAE of $0.063$. Both error metrics for the direct fit are approximately an order of magnitude larger than those of the induced set, highlighting its superior performance in approximating non-differentiable functions.

\subsection{Approximating functions with non-strict local extrema}

Last, we demonstrate the effectiveness of the induced set approximation for continuous functions with intervals of constancy. In particular, we look at a target function with a local extrema that is not strict, given by
\begin{align}
	\label{eq:f3}
	f(x) = 
	\begin{cases}
		\exp\left(-\frac{1}{(x-1)^{2}}\right), \quad &\abs{x} > 1 \\
		0, \quad &\abs{x} \leq 1 \\
	\end{cases},
\end{align}
and is defined in the domain $\Omega = [-4,4]$. $f$ is continuous and a set of
local minimum $X^* = [-1, 1]$ is contained in its domain. No other extrema are
present in the domain of $f$. For this reason, a second-degree polynomial
composed with an invertible function $h$ can be used to approximate $f$. 

We used the strategy of Example~\Cref{example:closed-form}. We
selected the second-degree polynomial $p_2$  with coefficients $a_0=f|_{X^*} = 0$,
$a_1=0$ and $a_2=1$. This polynomial was then composed with an invertible
function $h_\Theta$, modelled by an iResNet. The parameters of the neural
network were trained using a grid of $1000$ equidistant  points over the
function domain in the same fashion as of~\eqref{eq:Objective_nonl}. 

\begin{figure}[h!]
	\includegraphics[width=\linewidth]{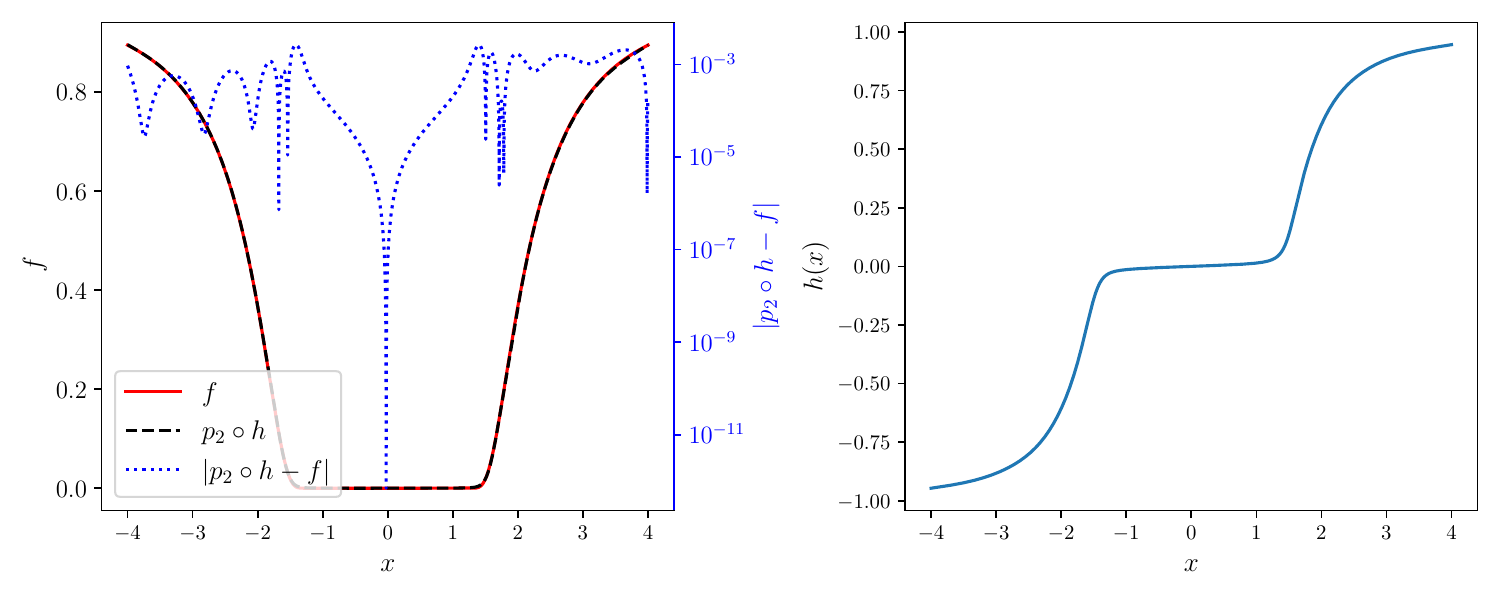}
	\caption{\textbf{Approximation of \eqref{eq:f3} using polynomials composed with a homeomorphism}. \textit{Left:} The target function (solid red) and the fitted second-order expansion
	composed with an invertible function (dashed black). The error of the fit is
	also plotted (dotted blue) for a visualization of the points where the
	fitting achieves higher and lower accuracy.
	\textit{Right:} The obtained invertible transformation $h$ for the fit.}
	\label{fig:minfty}
\end{figure}

The results were validated by comparison to a standard polynomial fit using the
training points and evaluated over a validation grid of $5000$ equidistant
points. The obtained RMSE for the induced polynomial on the validation set was
$9.40 \cdot 10^{-4}$, and the maximum absolute error of this fit was $2.67 \cdot
10^{-3}$. To achieve a comparative accuracy using a polynomial fit, a
polynomial of degree $40$ was required. A polynomial of degree 40 reaches an RMSE of $
9.16 \cdot 10^{-4}$ in the validation grid and a maximum deviation of $2.65
\cdot 10^{-3}$.  

The results of the approximation are shown in \Cref{fig:minfty}. 
In the set of local minimum, the second-order degree induced polynomial approximates the function $f$ exactly only at a single point $x_0 \in X^*$.
The error of the approximation increases as one moves away from
$x_0$ within $X^*$. The learned transformation
$h_{\Theta^*}$ acts by contracting $X^*$ into a
set of negligible measure, that can then be resolved by standard polynomials.

\section{Higher-dimensional applications}
\label{sec:multiD}

Thus far, we have demonstrated that a homeomorphism can induce
a dense set tailored to a specific univariate approximation problem. Under suitable
conditions on the target function, such a transformation enables an exact
finite-dimensional representation.

While these results are already powerful in one dimension, their full potential is realized in higher-dimensional applications. Consider the problem of approximating a target function 
\( f \in C(\Omega) \) over a multidimensional compact and connected domain \( \Omega \subset
\mathbb{R}^D \). A standard approach is to construct a multidimensional
polynomial basis
as the direct product of univariate monomials. Denoting the coordinates of the domain by \( x_\alpha \), \( \alpha = 1, \dots, D \), one typically seeks an approximation of \( f \) in the linear span of
\begin{equation}
	\label{eq:multi_d}
	\left\{ x_1^{i_1} x_2^{i_2} \dots x_D^{i_D} \right\}_{i_1, i_2, \dots, i_D = 0}^{N_D},
\end{equation}
which becomes dense in \( C(\Omega) \) as \( N_D \to \infty \). However, the
number of basis functions \( N_D \) required to achieve a given accuracy
increases exponentially with the dimension \( D \), a manifestation of the curse
of dimensionality. Consequently, constructing dense sets that are tailored to
the approximation problem at hand is of particular importance for high-dimensional
problems.

Such tailored dense sets can be constructed naturally through composition, as established in this work. Let
\begin{equation*}
	h: \Omega \to \Omega_h \subset \mathbb{R}^D
\end{equation*}
be a homeomorphism.
Define new variables \( q_k = h_k(\mathbf{x}) \), \( k = 1, \dots, D \), such that \( h \) performs a coordinate transformation from \( \mathbf{x} \) to \( \mathbf{q} \). Since the dimensionality of the domain was never used in its proof, \Cref{th:dense_set} can easily be extended to higher dimensions to show that the set
\begin{equation}\label{eq:multiDinduced}
	\left\{ q_1^{i_1} q_2^{i_2} \dots q_D^{i_D} \right\}_{i_1, i_2, \dots, i_D = 0}^{\infty}
\end{equation}
is dense in \( C(\Omega) \). 

In what follows, we demonstrate numerically the existence of functions $h$, such
that approximations in the span of $N$ functions of the set \eqref{eq:multiDinduced}
are more accurate than approximations in the span of \eqref{eq:multi_d}.

\subsection{2-D fitting}
Let $\Omega = [-4, 4] \times [-4, 4]$, and consider the target function 
\begin{equation}
\label{eq:f4}
	f : \Omega \to \mathbb{R}, \qquad f(x, y) = \arctan(x)\,\arctan(y),
\end{equation}
which is continuous on $\Omega$.

For fitting, the domain $\Omega$ was discretized using a tensor-product grid of $20$ equidistant points per dimension, and the target function was sampled at these locations. We fitted these data using a polynomial of degree~$2$, denoted by $p_2$, composed with a homeomorphism $h_\Theta$ modelled by an iResNet. The architecture of the iResNet remains the same as previously described in~\Cref{sec:NumRes}, but accepts a two-dimensional input and outputs two-dimensional values.  This construction creates an induced set consisting of six functions. The parameters $\Theta$ were trained to minimize the RMSE, and the linear coefficients of the polynomial expansion were obtained by pseudo-inversion of the evaluation matrix, both on the training grid. This procedure follows the same formulation as~\eqref{eq:Objective_nonl2}, extended to a two-dimensional domain.

For comparison, we fitted a standard polynomial of the form~\eqref{eq:multi_d} of degree~$13$, denoted by $p_{13}$, to the same training data. The polynomial $p_{13}$ comprises $105$ functions. Both models were evaluated on a denser validation grid of $100$ equidistant points per dimension. 

A quantitative summary of both approaches is provided in
\Cref{tab:2d_errors}. The fit based on the dense set induced by the learned
transformation $h_\Theta$ achieves substantially higher accuracy while employing
an order of magnitude fewer basis functions. In particular, both the RMSE and the MAE
are
approximately two orders of magnitude smaller than those of the conventional
polynomial fit, demonstrating the efficiency of the induced dense set in
capturing nonlinear dependencies with compact representations.

Both approximations are plotted in \Cref{fig:2d} against $f$.

\begin{table}[h!]
    \begin{tabular}{lcccc}
        \toprule
        \textbf{Model} & \textbf{Degree} & \textbf{\# Basis Functions} & \textbf{RMSE} & \textbf{MAE} \\
        \midrule
        Polynomial composed with $h_\Theta$ ($p_2 \circ h_\Theta$) & $2$ & $6$ & $3.1 \cdot 10^{-4}$ & $0.0011$ \\
        Standard polynomial ($p_{13}$) & $13$ & $105$ & $2.3 \cdot 10^{-2}$ & $0.163$ \\
        \bottomrule
    \end{tabular}
    \vspace{2.5pt}
    \centering
    \caption{Comparison of two-dimensional polynomial fitting of the target function \eqref{eq:f4}.}
    \label{tab:2d_errors}
\end{table}

\begin{figure}[h!]
	\includegraphics[width=\textwidth]{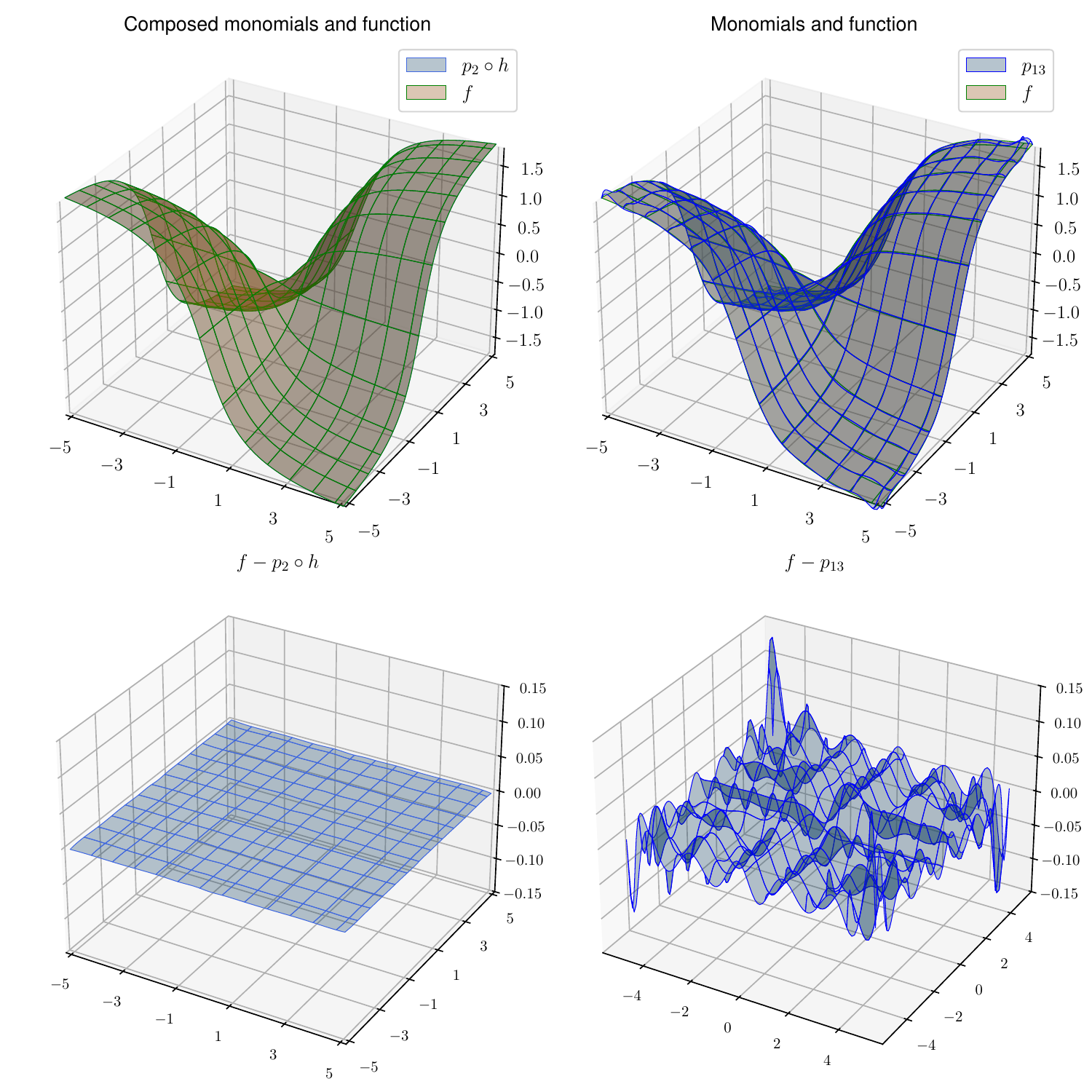}
	\caption{\textbf{Approximation of \eqref{eq:f4} using polynomials and
	polynomials composed with a homeomorphism. }\textit {Top left:} Plotted are the target function and the fitted
	second-order expansion composed with a homeomorphism. \textit{Bottom left:}
	Plotted is the
	difference between the target function and the induced polynomial on the
	validation set. \textit{Top
	right:} Plotted are the target function and the fitted $13$-degree polynomial.  \textit{Bottom
	right:} Plotted is the difference between the target function and the $13$-degree
	polynomial on the validation set.}
	\label{fig:2d}
\end{figure}

\subsection{Potential energy surface fitting}\label{sec:PES}

We now illustrate the utility of dense sets of the form~\eqref{eq:multiDinduced} for constructing high-dimensional surfaces required for solving differential equations arising in molecular structure theory. For a given molecular system, we first consider its electronic Schrödinger equation
\begin{equation}
    \label{eq:ETISE}
    (V + T)\, \psi_i (\mathbf{y};\mathbf{x}) = E_i(\mathbf{x})\, \psi_i (\mathbf{y};\mathbf{x}),
\end{equation}
where $\mathbf{y} \in \mathbb{R}^L$ denotes the electronic coordinates, $\mathbf{x} \in \Omega \subseteq \mathbb{R}^D$ the nuclear coordinates, $V$ is the potential operator describing static interparticle interactions, and $T$ is a second-order differential operator representing the electronic kinetic energy. Equation~\eqref{eq:ETISE} constitutes an infinite-dimensional eigenvalue problem, and the computation of its smallest eigenvalue $E_0(\mathbf{x})$ is central to many applications in molecular physics and chemistry.

The smallest eigenvalue $E_0 : \Omega \to \mathbb{R}$ is a real-valued continuous function that depends parametrically on the nuclear geometry
$\mathbf{x}$.  The representation of $E_0$ as a function of the nuclear geometry
$\mathbf{x}$ is called the potential energy surface (PES). Solving~\eqref{eq:ETISE} for each nuclear geometry is computationally
expensive. In practice, one solves it at a finite set of geometries to obtain a
dataset
\begin{equation*}
    S := \{ (\mathbf{x}_i, E_{0,i}) \}_{i=1}^{N},
\end{equation*}
from which $E_0(\mathbf{x})$ at new configurations is inferred by interpolation and extrapolation. A continuous representation of the PES is desired for solving the nuclear Schrödinger equation and its accuracy impacts the calculations of spectroscopic and dynamical properties of molecules.
Numerous interpolation and fitting strategies have been developed, including traditional
polynomial expansions~\cite{Jensen:JMolSpec128:478, Yachmenev:JCP134:244307,
Jensen:JMolSpec133:438, Tyuterev:CPL348:223} and modern machine-learning
approaches such as neural network potentials~\cite{Morawietz:JCP136:064103,
Natarajan:PCCP17:8356, Manzhos:IJQC115:1012, Saleh:JCP155:144109}. A comprehensive review can be
found in~\citeauthor{Manzhos:CR121:10187}~\cite{Manzhos:CR121:10187}.

Nuclear geometries are often expressed in internal coordinates, which describe relative positions of nuclei and are naturally connected to molecular vibrations. Direct interpolation of the PES in these coordinates can, however, be suboptimal. A common remedy is to introduce nonlinear transformations
\begin{equation}
    \label{eq:PES_transform}
    y_i = g_i(x_i),
\end{equation}
for all $i = 1, \dots, D$, where $g_i : \Omega_i \to \mathbb{R}$, and $\Omega_i$ is the domain of $x_i$. 
Then, the PES is approximated as a polynomial expansion in the transformed variables,
\begin{equation}
    \label{eq:PES_series}
    f(\mathbf{x}) = \sum_{i_1 + \dots + i_D = 0}^{N} c_{i_1, \dots, i_D}\,
    y_1^{i_1}(\mathbf{x}) \cdots y_D^{i_D}(\mathbf{x}),
\end{equation}
where $c_{i_1, \dots, i_D}$ are the expansion coefficients.  

In the PES literature, the functions $g_i$ are chosen based on expert intuition and are always strictly monotonic. However, to the best of our knowledge, no formal justification for this constraint has been provided, beyond empirical success. We propose that the underlying rationale can be understood through the lens of the density result presented in~\Cref{th:dense_set}. Specifically, we refer to the fact that composing the multidimensional polynomials with a homeomorphism yields a set that remains dense. This provides a principled explanation for the enhanced convergence of the power series in the transformed variables $\mathbf{y}$, particularly when the transformation encapsulates the behavior of the PES. 

The class of coordinate-wise strictly monotonic transformations $g_i$ in~\eqref{eq:PES_transform} can be viewed as a subclass of the family of more general multidimensional homeomorphisms $\mathbf{q} = h(\mathbf{x})$, where $h: \Omega \to \Omega_h$, with $\Omega_h \subseteq \mathbb{R}^D$. In this broader setting, the transformation is constructed jointly across all coordinates. The PES can then be approximated as a power series of the variable $\mathbf{q}$, following the same structure as~\eqref{eq:PES_series}.  
Here, we propose to learn such a transformation directly, yielding a dense set induced that is specifically adapted to the target PES.

We illustrate the potential of this approach by fitting the PES of the H$_2$S molecule. In this example, the internal coordinates $\mathbf{x} = (x_0, x_1, x_2)$ correspond to
the two
H-S bond lengths and the H-S-H bond angle. The radial
coordinates span $[0,\infty)$ and the angular coordinate $[0,\pi]$. In practice,
the radial domains are truncated to finite intervals $[a,b]$ with $b > a > 0$,
restricting attention to physically relevant configurations and yielding a
compact domain suitable for polynomial approximation.

As reference data, we used the analytical PES of~\citeauthor{Azzam:MNRAS460:4063}~\cite{Azzam:MNRAS460:4063} to generate synthetic samples. Each coordinate was uniformly sampled, and a tensor-product grid of $40$ equidistant points per coordinate was used to evaluate the reference values of the PES. For the radial coordinates, we sampled the interval $[0.9, 3.5]$ \AA, and for the bond angle, $[0,\pi]$. Points corresponding to high-energy regions ($V > 4 \cdot 10^4~\text{cm}^{-1}$) were discarded to avoid numerical instabilities of the previous fitting. The potential minimum was explicitly included to ensure accurate reproduction of the constant term.  
A validation set was generated using a denser grid of $100$ equidistant points per coordinate, constructed analogously.

The original PES was built using the transformed variables defined as
\begin{align*}
    y_0 &= 1 - \exp[-\alpha_0 (x_0 - \beta_0)], \\
    y_1 &= 1 - \exp[-\alpha_1 (x_1 - \beta_1)], \\
    y_2 &= \cos(x_2) - \cos(\beta_2),
\end{align*}
where $\{\alpha_i\}_{i=0,1}$ control the width of the PES minima and $\{\beta_j\}_{j=0,1,2}$ specify the equilibrium configuration. We remark that the three transformations are indeed strictly monotonic in their domain.

We modeled the transformation $h_\Theta$ using an iResNet (using the same architecture previously described, but $3$ dimensions into $3$ dimensions) and considered the induced variables
\begin{equation*}
    \mathbf{q} = h_\Theta(\mathbf{x}).
\end{equation*}
The PES was then approximated using the polynomial expansion~\eqref{eq:PES_series} in the learned variables.

For comparison, we also fitted a conventional polynomial directly in the internal coordinates $\mathbf{x}$ using the same dataset.  
Table~\ref{tab:pes_results} summarizes the obtained results.

\begin{table}[h!]

	\begin{tabular}{lcccc}
	\toprule
	\textbf{Variable representation} & \textbf{Polynomial degree} & \textbf{\# functions} & \textbf{RMSE ($\text{cm}^{-1}$)} & \textbf{MAE ($\text{cm}^{-1}$)} \\
	\midrule
	Learned variables $\mathbf{q} = h_\Theta(\mathbf{x})$ & 4 & 35 & $18.93$ & $219.18$ \\
	Internal coordinates $\mathbf{x}$ & 18 & 1330 & $105.54$ & $2216.09$ \\
	\bottomrule
	\end{tabular}
	\vspace{2.5pt}
	\centering
	\caption{Comparison of PES fitting performance for H$_2$S using different variable representations. Errors are reported on the validation grid.}
	\label{tab:pes_results}
\end{table}

The dense set induced by the learned homeomorphism \(h_\Theta\) yields an order-of-magnitude improvement in fitting accuracy while requiring nearly two orders of magnitude fewer functions. This highlights the efficiency and enhanced convergence of dense sets adapted to the underlying structure of the potential energy surface.

\section{Conclusions}
In this work, we developed a theory for constructing families of dense sets in the space of continuous functions by composing known dense sets with homeomorphisms. 

We highlighted the advantages of such induced dense sets for approximation
problems. In \Cref{th:Nextreme}, we established the existence of a finite
dimensional approximation of any univariate continuous target function with
finitely many sets of local extrema. This approximation was constructed by composing a
polynomial with a homeomorphism. The
degree of the polynomial is determined by the number of maxima and minima of the target function. All
theoretical claims were validated through numerical experiments using iResNets
to model the function \( h \), demonstrating orders-of-magnitude improvements in
accuracy over standard polynomial approximations.

Similarly, dense sets on higher-dimensional domains can be generated by composing multivariate polynomials
with homeomorphisms. We illustrated the benefits of these induced dense sets by
fitting the potential energy surface of the H\(_2\)S molecule, showing that the
induced set achieves substantially higher accuracy with at least an order of
magnitude fewer terms than a direct polynomial fit.

Overall, the proposed framework provides a new way to improve the accuracy of approximating continuous functions: rather than increasing the dimension of the approximation space within a fixed dense set, one can optimize the choice of the dense set itself to achieve more efficient and accurate approximations.
\section*{Acknowledgments}
The authors would like to acknowledge Emil Vogt, Jochen Küpper and the CFEL
Controlled Molecule Imaging group for valuable scientific discussion. The
authors would like to acknowledge David L. Bishop for suggesting a reference
that was central to the proof of~\Cref{Thm:Chandler}.

This work was supported by Deutsches Elektronen-Synchtrotron DESY, a member of the Helmholtz Association (HGF), including the Maxwell computational resource operated at DESY, by the Data Science in Hamburg HELMHOLTZ Graduate School for the Structure of Matter (DASHH, HIDSS-0002), and by the Deutsche Forschungsgemeinschaft (DFG) through the cluster of excellence ``Advanced Imaging of Matter'' (AIM, EXC~2056, ID~390715994).

\section*{Author contributions}

The project was conceived by A.F.C. The theoretical development and
formalization of the main results were carried out jointly by A.F.C. and Y.S.
Numerical computations and figure generation were performed by A.F.C. Both
authors contributed to the interpretation of the results and to writing the
manuscript.

\section*{Competing interests}
The authors declare that they have no conflict of interests.
\section*{Declaration of generative AI and AI-assisted technologies in the manuscript preparation process}
During the preparation of this work the authors used large language models in
order to improve the language and flow in some paragraphs. After using these models, the authors reviewed and
edited the content as needed and take full responsibility for the content of
the published article.
\printbibliography
\end{document}